\documentclass[10pt]{amsart}
\usepackage[dvips]{graphicx}
\usepackage{pinlabel}

\theoremstyle{plain}
\newtheorem{thm}{Theorem}[section]
\newtheorem{lem}[thm]{Lemma}

\theoremstyle{definition}
\newtheorem*{defn}{Definition}

\newcommand{\TL}{{\mathcal{TL}}}
\newcommand{\id}{\mathrm{id}}

\newcommand{\co}{\colon\thinspace}

\title{Skein theory for the $ADE$ planar algebras}
\author{Stephen Bigelow}
\address{Department of Mathematics,
         University of California at Santa Barbara,
         California 93106, USA}
\email{bigelow@math.ucsb.edu}
\date{January 2009}

\begin{document}

\begin{abstract}
We give generators and relations for the planar algebras
corresponding to $ADE$ subfactors.
We also give a basis 
and an algorithm to express an arbitrary diagram
as a linear combination of these basis diagrams.
\end{abstract}

\maketitle

%=============================================================
\section{Introduction}

The notion of a planar algebra is due to Jones \cite{jonesplanar}.
The roughly equivalent notion of a spider
is due to Kuperberg \cite{kuperberg}.
Planar algebras arise in many contexts
where there is a reasonably nice category
with tensor products and duals.
Examples are the category of representations of a quantum group,
or of bimodules coming from a subfactor.

The subfactor algebras of index less than $4$ can be classified into
the two infinite families $A_N$ and $D_{2N}$,
and the two sporadic examples $E_6$ and $E_8$.
See
\cite{ghj},
\cite{ocneanu},
\cite{izumi},
and
\cite{ko}
for the story of this ``ADE'' classification.

The {\em Kuperberg program} can be summarized as follows.
    \begin{quote}
    Give a presentation for every interesting planar algebra,
    and prove as much as possible about the planar algebra
    using only its presentation.
    \end{quote}
The planar algebras corresponding to subfactors of type $A_N$
are fairly well understood.
In \cite{mps},
Morrison, Peters and Snyder basically complete the Kuperberg program
in the $D_{2N}$ case.

The aim of this paper
is to extend the results of \cite{mps}
types $E_6$ and $E_8$.
However our approach is different.
Whereas \cite{mps}
use only combinatorial arguments starting from their presentation,
we will rely on the existence of the subfactor planar algebra
and some of its known properties.
Thus this paper is not completely in the spirit of the Kuperberg program.
As a compensation,
we will address the following program,
suggested by Jones in Appendix~B of \cite{jonesannular}.
    \begin{quote}
    Give a basis for every interesting planar algebra,
    and an algorithm to express any given diagram
    as a linear combination of basis elements.
    \end{quote}

Most of this paper concerns the subfactor of type $E_8$.
In Section~\ref{sec:Eeight} we define two planar algebras.
The planar algebra $\mathcal{P}$ is defined
by a presentation in terms of one generator and five relations.
The planar algebra $\mathcal{P}'$ is
the subfactor planar algebra whose principal graph is the $E_8$ graph.
We give the properties of $\mathcal{P}'$ that we need.
In Section~\ref{sec:isomorphism}
we prove that $\mathcal{P}$ is isomorphic to $\mathcal{P}'$.
In Section~\ref{sec:basis}
define a set of diagrams that will form a basis for our planar algebra.
The proof that the basis spans is constructive,
although we have not tried to give an efficient algorithm.
Finally,
in Section~\ref{sec:ade},
we explain how our methods can be applied to types $E_6$,
$A_N$ and $D_{2N}$.

%----------------------
\section{Planar algebras}

We give a brief and impressionistic review of
the definition of a planar algebra.
For the details,
see the preprint \cite{jonesplanar} at Vaughan Jones' website.

Something that is called an ``algebra''
is usually a vector space
together with one or more additional operations.
A planar algebra $\mathcal{P}$ consists of infinitely many vector spaces
(or one graded vector space if you prefer),
together with infinitely many operations.
For every non-negative integer $k$,
we have a vector space $\mathcal{P}_k$.
For every planar arc diagram $T$,
we have a multilinear $n$-ary operation
$$\mathcal{P}(T) \co 
  \mathcal{P}_{k_1} \otimes \dots \otimes \mathcal{P}_{k_n} \to
  \mathcal{P}_{k_0},$$
where $n$ is the number of internal disks in $T$,
$k_1,\dots,k_n$ are the numbers of 
ends of strands on the internal disks of $T$,
and $k_0$ is the number of ends of strands on the external disk.

%EXAMPLE???

In practice,
$\mathcal{P}_k$ will be spanned by some kind of diagrams in a disk.
Each diagram in $\mathcal{P}_k$ has $k$ endpoints of strands on its boundary.
The action of a planar arc diagram $T$
is given by gluing diagrams into the interior disks of $T$,
matching up endpoints of strands on the boundary of the diagrams
with the endpoints of strands in $T$.
To determine ``which way up'' to glue the diagrams,
we use basepoints on the boundaries of diagrams
and on the boundaries of input disks of $T$.
These basepoints are indicated by a star,
and are never allowed to coincide with the endpoints of strands.

Note that,
for ease of exposition,
we only work with ``unshaded'' planar algebras.

%-------------------------------------------
\subsection{Composition}

It is often convenient to draw an element of a planar algebra
in a rectangle instead of a round disk.
When we do this,
the basepoint will always be at the top left corner,
and the endpoints of strands will be on the top and bottom edges.

Let $\mathcal{P}^a_b$ denote the elements of $\mathcal{P}_{a+b}$,
drawn in a rectangle,
with $a$ endpoints at the top
and $b$ at the bottom.
If $A \in \mathcal{P}^a_b$
and $B \in \mathcal{P}^b_c$
then the {\em composition} of $A$ and $B$
is the element $AB \in \mathcal{P}^a_c$
obtained by stacking $A$ on top of $B$.
Note that the meaning of this composition depends on the value of $b$,
which must be made clear from context.
(Here,
we are blurring the distinction between the planar algebra
and the corresponding category,
as defined in \cite{mps}.)

%-------------------------------------------
\subsection{Quantum integers}

Suppose $q$ is a non-zero complex number.
The {\em quantum integer} $[n]$ is given by
$$[n] = \frac{q^n - q^{-n}}{q - q^{-1}}.$$
These appear only briefly in this paper,
and they can be treated as constants
whose precise value is unimportant.
However they play an important role behind the scenes,
for example in the definition of the Jones-Wenzl idempotents,
and in the properties of the subfactor planar algebra.

%-------------------------------------------
\subsection{Temperley-Lieb planar algebra}

A {\em Temperley-Lieb diagram} is
a finite collection of disjoint properly embedded edges in a disk,
together with a basepoint on the boundary.
These form a planar algebra as follows.
Suppose $T$ is a planar arc diagram with $n$ holes,
and $D_1,\dots,D_n$
are Temperley-Lieb diagrams with the appropriate number of endpoints.
We can create a new Temperley-Lieb diagram $A$
by inserting $D_1,\dots,D_n$ into the holes in $T$,
and deleting any resulting strands that form closed loops.
Let $m$ be the number of closed loops that were deleted.
Then $T$ maps the $n$-tuple $D_1,\dots,D_n$ to $[2]^m A$.

The planar algebra of Temperley-Lieb diagrams
is called the {\em Temperley-Lieb planar algebra},
and will be denoted $\TL$.
It can be defined more briefly
as the planar algebra with no generators
and a single defining relation
$$\raisebox{-10pt}{\includegraphics[scale=0.6]{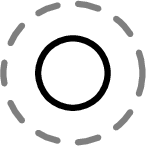}} = [2]  \,
  \raisebox{-10pt}{\includegraphics[scale=0.6]{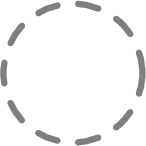}}.$$
Every planar algebra we consider will satisfy the above relation,
and hence contain an image of $\TL$.

We now list some important examples of Temperley-Lieb diagrams.
The {\em identity diagram} $\id_n \in \TL^n_n$
is the diagram consisting of $n$ vertical strands in a rectangle.

Suppose $D \in \TL^m_n$ is a Temperley-Lieb diagram drawn in a rectangle.
We will say $D$ {\em contains a cup}
if it contains a strand
that has both endpoints on the top edge of the rectangle.
We say $D$ {\em is} a cup
if it is consists of $n$ vertical strands
and one strand that has both endpoints on the top of the rectangle.

Similarly,
a {\em cap} is a diagram in $\TL^n_{n+2}$
that has $n$ vertical strands
and one strand with both endpoints on the bottom of the rectangle.

The {\em Jones-Wenzl idempotent} $P_n$
is unique element of $\TL^n_n$ with the following properties.
\begin{itemize}
\item When $P_n$ is expressed as
      a linear combination of Temperley-Lieb diagrams,
      the diagram $\id_n$ occurs with coefficient $1$.
\item If $X \in \TL^{n-2}_n$ is any cap then $X P_n = 0$.
\item If $Y \in \TL^n_{n-2}$ is any cup then $P_n Y = 0$.
\end{itemize}
In all of our examples,
$q$ will be of the form $e^{i \pi / N}$.
In this case the element $P_n$ exists and is unique for all $n \le N-1$,
but does not exist for $n \ge N$.

Let a {\em crossing} be the following element of $\mathcal{P}_4$.
$$
% crossing:
     \labellist
       \tiny \hair 1pt
       \pinlabel $\star$ [l] at 0 50
     \endlabellist
     \raisebox{-20pt}{\includegraphics[scale=0.4]{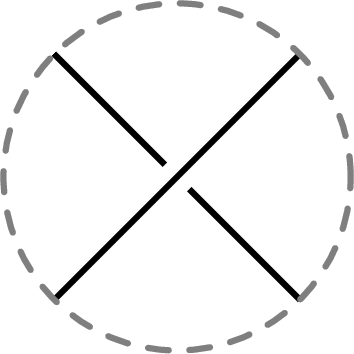}}
= iq^{\frac{1}{2}}
% crossing resolved vertically:
     \labellist
       \tiny \hair 1pt
       \pinlabel $\star$ [l] at 0 50
     \endlabellist
     \raisebox{-20pt}{\includegraphics[scale=0.4]{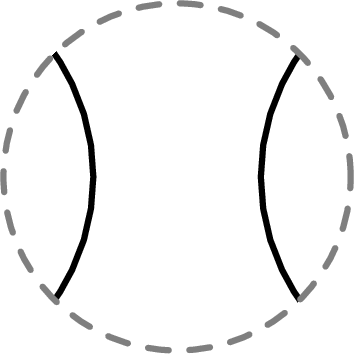}}
- iq^{-\frac{1}{2}}
% crossing resolved vertically:
     \labellist
       \tiny \hair 1pt
       \pinlabel $\star$ [l] at 0 50
     \endlabellist
     \raisebox{-20pt}{\includegraphics[scale=0.4]{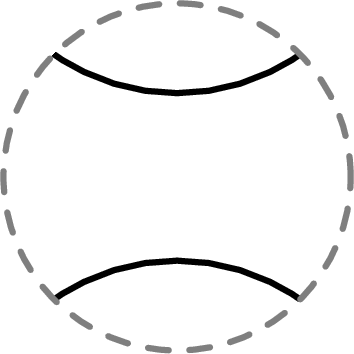}}
$$

The crossing allows us to consider knot and tangle diagrams
as representing elements of the Temperley-Lieb planar algebra.
We can express a diagram with $k$ crossings
as a linear combination of $2^k$ diagrams that have no crossings.
This process is called
{\em resolving the crossings}.

The crossing satisfies Reidemeister moves two and three.
In place of Reidemeister one,
we have the following.
$$
\raisebox{-20pt}{\includegraphics[scale=0.4]{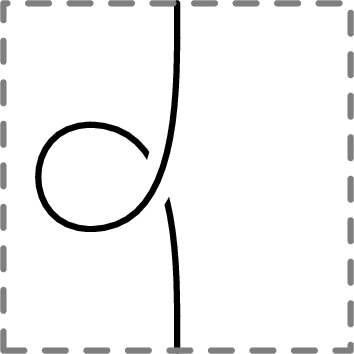}} = iq^{\frac{3}{2}}
\raisebox{-20pt}{\includegraphics[scale=0.4]{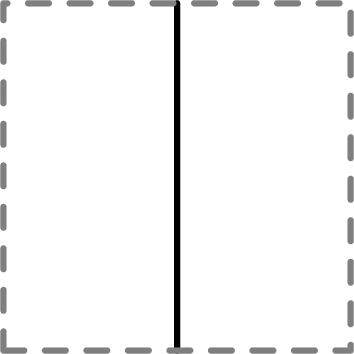}}.
$$

%--------------------------
\section{The $E_8$ planar algebra}
\label{sec:Eeight}

The purpose of this section
is to define the planar algebras $\mathcal{P}$ and $\mathcal{P}'$.
The first is given by a presentation,
and the second is the subfactor planar algebra
whose principal graph is the $E_8$ graph.
In the next section,
we will show that $\mathcal{P}$ and $\mathcal{P}'$ are isomorphic.

%---------------
\subsection{The presentation}

We define $\mathcal{P}$
in terms of generators and relations.
There is just one generator $S \in \mathcal{P}_{10}$.

Before we list the relations,
we define some notation.
Let $q = e^{i \pi / 30}$.
Let $\omega = e^{6 i \pi / 5}$.
Let $\rho(S)$,
$\tau(S)$,
$S^2$
and $\hat{S}$
be as shown in Figure~\ref{fig:rhoetc}.
We call $\rho(S)$ the {\em rotation} of $S$
and $\tau(S)$ the {\em partial trace} of $S$.
\begin{figure}[htb]
$$
     \labellist
     \tiny \hair 1pt
     \pinlabel $S$ at 50 50
     \pinlabel $\star$ [r] at 30 50
     \endlabellist
     \rho(S) = \raisebox{-18pt}{\includegraphics[scale=0.4]{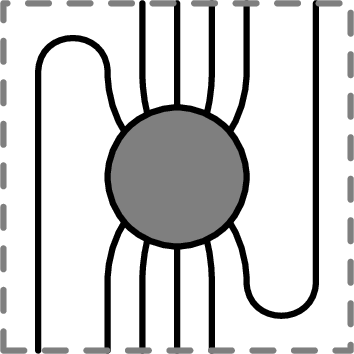}},
\quad
     \labellist
     \tiny \hair 1pt
     \pinlabel $S$ at 50 50
     \pinlabel $\star$ [r] at 30 50
     \endlabellist
     \tau(S) = \raisebox{-18pt}{\includegraphics[scale=0.4]{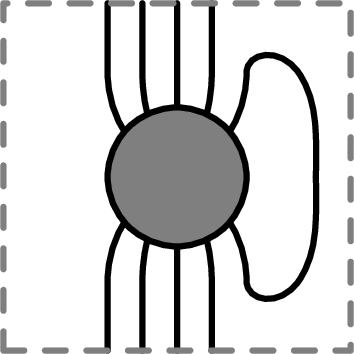}},
\quad
     \labellist
     \tiny \hair 1pt
     \pinlabel $S$ at 50 25
     \pinlabel $S$ at 50 75
     \pinlabel $\star$ [r] at 35 25
     \pinlabel $\star$ [r] at 35 75
     \endlabellist
     S^2 = \raisebox{-18pt}{\includegraphics[scale=0.4]{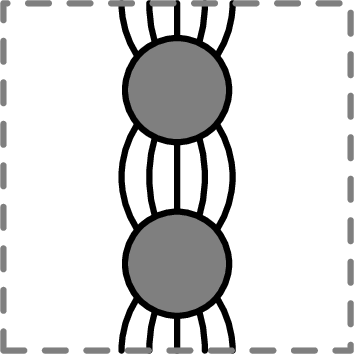}},
\quad
     \labellist
     \tiny \hair 1pt
     \pinlabel $S$ at 50 50
     \pinlabel $\star$ [b] at 50 70
     \endlabellist
     \hat{S} = \raisebox{-18pt}{\includegraphics[scale=0.4]{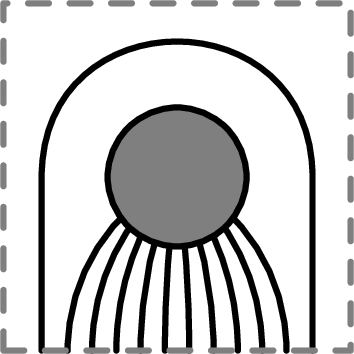}}
$$
\caption{The diagrams $\rho(S)$, $\tau(S)$, $S^2$ and $\hat{S}$, respectively}
\label{fig:rhoetc}
\end{figure}

The defining relations of $\mathcal{P}$ are as follows.
\begin{itemize}
\item $\raisebox{-6pt}{\includegraphics[scale=0.4]{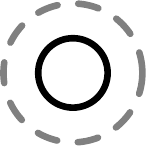}} = [2]  \,
       \raisebox{-6pt}{\includegraphics[scale=0.4]{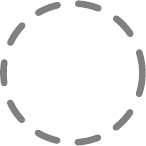}}$,
\item $\rho(S) = \omega S$,
\item $\tau(S) = 0$,
\item $S^2 = S + [2]^2[3] P_5$,
\item $\hat{S} \, P_{12} = 0$.
\end{itemize}

We call the first four relations
the {\em bubble bursting},
{\em chirality},
{\em partial trace},
and {\em quadratic} relation,
respectively.
The fifth relation is equivalent to the following
{\em braiding relation}.

\begin{lem}
\label{lem:braid}
$
     \labellist
     \tiny \hair 1pt
     \pinlabel $S$ at 50 50
     \pinlabel $\star$ [b] at 50 70
     \endlabellist
  \raisebox{-18pt}{\includegraphics[scale=0.4]{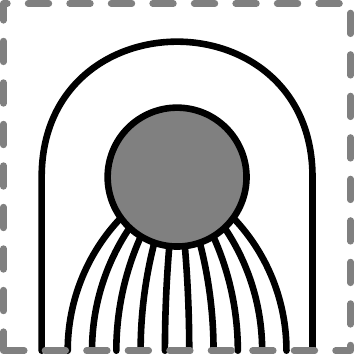}}
=
     \labellist
     \tiny \hair 1pt
     \pinlabel $S$ at 50 50
     \pinlabel $\star$ [b] at 50 70
     \endlabellist
  \raisebox{-18pt}{\includegraphics[scale=0.4]{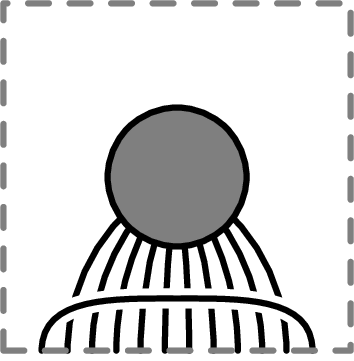}}.$
\end{lem}

\begin{proof}
Let $X$ denote the diagram
on the right hand side of the braiding relation.
We must show that $\hat{S} = X$.

Suppose $Y$ is the cup
$$ Y = \raisebox{-9pt}{\includegraphics[scale=0.4]{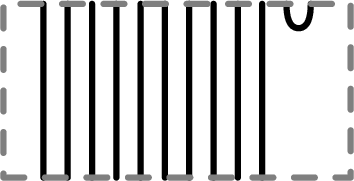}}.$$
Then
$$\hat{S} Y = \rho(S) = \omega S.$$
To compute $XY$,
first apply Reidemeister one to the rightmost crossing,
giving a factor of $iq^{3/2}$.
Now resolve the remaining nine crossings.
All but one of the resulting terms
contains a cup connected directly to $S$,
so can be eliminated.
The easiest way to see this is to resolve the crossings one by one,
working from right to left.
In the end,
we obtain
$$XY = (iq^{3/2})(iq^{1/2})^9 S = -q^6 S.$$
By comparing the above expressions,
we find that $\hat{S}Y = XY$.

Similarly,
if $Y$ is the cup
$$ Y = \raisebox{-9pt}{\includegraphics[scale=0.4]{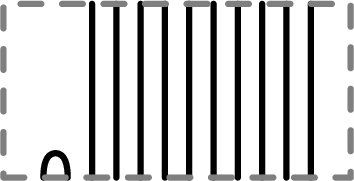}}.$$
then $\hat{S}Y = XY$.
If $Y$ is any other cup
then $\hat{S}Y$ and $XY$ are both zero.

Recall that $P_{12}$ is a linear combination of Temperley-Lieb diagrams.
Every term in this linear combination contains a cup,
except for the identity,
which occurs with coefficient one.
Thus
$$\hat{S} = X \; \Leftrightarrow \; \hat{S} P_{12} = X P_{12}.$$
But $X P_{12}$ is clearly zero.
Thus the desired result $\hat{S} = X$
is equivalent to the relation $\hat{S} P_{12} = 0$.
\end{proof}

The braiding relation
says that it is possible to pass a strand over the generator $S$.
However it is not possible to pass a strand under $S$,
so $\mathcal{P}$ is not a {\em braided planar algebra}.
It is not even possible to pass a strand under $S$ up to a change of sign,
so $\mathcal{P}$ is not {\em partially braided}
in the sense of \cite[Theorem~3.2]{mps}.
However the above braiding relation is enough for many purposes.

% WHY NOT USE THE BRAIDING RELATION AS A DEFINING RELATION?

%-------------------------------
\subsection{The subfactor planar algebra}
\label{subsec:subfactor}

The $E_8$ graph is as follows.
$$\raisebox{-8pt}{\includegraphics[scale=1]{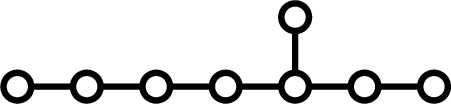}}$$
Let $\mathcal{P}'$ be
the subfactor planar algebra with principal graph $E_8$.
This exists and is unique,
up to some choices of convention.
% REFERENCE???

We now give some known properties of $\mathcal{P}'$.
The proofs require some background knowledge
on subfactor planar algebras and principal graphs.
The meaning of these terms is explained in \cite{mps}
and \cite{jonesplanar}.

\begin{lem}
\label{lem:surjection}
$\mathcal{P}'$ is generated by an element $R \in \mathcal{P}'_{10}$.
Up to some choice of conventions,
$R$ satisfies the defining relations of $S$.
\end{lem}

\begin{proof}
We take $R$ to be a generator for
the space of morphisms from $P_{10}$ to the empty diagram.

The constant $[2]$ in the bubble bursting relation
is the positive eigenvalue of the $E_8$ graph.

The chirality relation is given in \cite[Theorem~4.2.13]{jonesplanar}.
The ``chirality'' in that theorem
is the value of $\omega^2$ in our chirality relation.
The shaded planar algebra is only unique up to complex conjugation.
The unshaded version also requires
an arbitrary choice of sign of $\omega$.
We chose a value of $\omega$
so that the braiding relation hold
with our definition of a crossing.

The partial trace relation
is immediate from our choice of $R$.

The quadratic relation is equation 4.3.5 in \cite{quadtangle}.
There is currently a misprint:
it should read $R^2 = (1-r)R + rp_n$.
However the proof is correct.
In our context,
$r = [2]^2/[3]$,
and we have rescaled $R$ by a factor of $[3]$.

The relation $\hat{S} P_{12} = 0$
follows from the fact that
there is no non-zero morphism from $P_{12}$ to the empty diagram
in $\mathcal{P}'$.
\end{proof}

\begin{lem}
\label{lem:seven}
In $\mathcal{P}'$,
$id_7$ is equal to
a linear combination of diagrams of the form $AB$
such that $A \in (\mathcal{P}')^7_m$,
and       $B \in (\mathcal{P}')^m_7$,
for some $m < 7$.
\end{lem}

\begin{proof}
This comes down to
the fact that the $E_8$ graph has diameter less than seven.
Every minimal idempotent
is a summand of $\id_m$ for some $m < 7$.
\end{proof}

\begin{lem}
\label{lem:jw}
$P_{29} = 0$ in $\mathcal{P}'$.
\end{lem}

\begin{proof}
This holds in any subfactor planar algebra where $[30] = 0$.
\end{proof}

We now give the dimension of $\mathcal{P}'_n$
in a form that will be useful later.

\begin{defn}
Suppose $Y \in \TL^m_n$ is a Temperley-Lieb diagram drawn in a rectangle.
Let $x_0,\dots,x_n$ be a sequence of points on the bottom edge of $Y$,
ordered from left to right,
occupying the $n+1$ spaces between the endpoints of strands.
We say $Y$ is {\em JW-reduced} if
$Y$ contains no cups,
and for all $i = 0,\dots,n$,
there are at most $28$ strands
that have one endpoint to the left of $x_i$
and the other endpoint either to the right of $x_i$ or on the top edge of $Y$.
\end{defn}

\begin{lem}
\label{lem:dimension}
For $m,n \ge 0$,
let $d^m_n$ be the number of JW-reduced diagrams in $\TL^m_n$.
Then $\dim (\mathcal{P}'_n) = d^0_n + d^{10}_n + d^{18}_n + d^{28}$.
\end{lem}

\begin{proof}
The idea is to decompose $\id_n$
into a direct sum of minimal idempotents
in the category corresponding to $\mathcal{P}'$.
The dimension of $\mathcal{P}'_n$
is the number of copies of the empty diagram in this decomposition.

First we work in the image of the Temperley-Lieb algebra in $\mathcal{P}'$.
This is the Temperley-Lieb algebra
satisfying the bubble bursting relation
and $P_{29} = 0$.
The JW-reduced diagrams in $\TL^m_n$
give a basis for the space of morphisms from $\id_n$ to $P_m$.
Thus $d^m_n$ is the number of copies of $P_m$
in the decomposition of $\id_n$.

Now decompose $P_m$ into minimal idempotents,
working in $\mathcal{P}'$.
This is easy to do using the information encoded in the principal graph.
The number of copies of the empty diagram
in the decomposition of $P_m$
is one if $m \in \{0,10,18,28\}$
and zero otherwise.

Combining these two facts,
we see that the number of copies of the empty diagram
in the decomposition of $\id_n$ into minimal idempotents
is as claimed.
\end{proof}

%---------------------------------------
\section{The isomorphism}
\label{sec:isomorphism}

The aim of this section is to prove the following.

\begin{thm}
\label{thm:iso}
$\mathcal{P}'$ is isomorphic to $\mathcal{P}$
\end{thm}

By Lemma~\ref{lem:surjection},
there is a surjective planar algebra morphism $\Phi$
from $\mathcal{P}$ to $\mathcal{P}'$,
taking $S$ to $R$.
It remains to show that $\Phi$ is injective.

\begin{lem}
\label{lem:zero}
Every element of $\mathcal{P}_0$
is a scalar multiple of the empty diagram.
\end{lem}

\begin{proof}
Suppose $D$ is a diagram in $\mathcal{P}_0$.
We must show that $D$ is a scalar multiple of the empty diagram.
Let $m$ be the number of copies of $S$ in $D$.
We will use induction on $m$.
By the braiding relation,
we can assume the copies of $S$ lie on
the vertices of a regular $m$-gon,
and that every strand lies inside this $m$-gon.

By resolving all crossings,
we can reduce to the case that $D$ is a diagram with no crossings.  
By the bubble bursting relation,
we can assume $D$ contains no closed loops.
By the chirality and partial trace relations,
we can assume $D$ contains no strand
with both endpoints on the same copy of $S$.
Thus every strand in $D$ connects
two distinct copies of $S$.

We can think of $D$ as a triangulated $m$-gon,
where the edges have multiplicities.
(We may need to add some edges with multiplicity zero
if we want the edges to cut the $m$-gon into triangles.)
Any triangulated polygon
has a vertex with valency two,
not counting multiplicities.
However this vertex has valency $10$ if we count multiplicities.
Thus there is an edge with multiplicity at least $5$.
This gives us a copy of $S^2$ inside $D$,
up to rotation of the copies of $S$.
The result now follows from
the quadratic relation and the induction hypothesis.
\end{proof}

\begin{lem}
\label{lem:ten}
$\mathcal{P}_{10}$ is spanned by
Temperley-Lieb diagrams and $S$.
\end{lem}

\begin{proof}
Suppose $D$ is a diagram in $\mathcal{P}^0_{10}$.
We must show that $D$ is a linear combination
of Temperley-Lieb diagrams and $S$.
We use induction on the number of copies of $S$ in $D$.

By the braiding relation,
we can assume the copies of $S$
lie in a row at the top of $D$,
and all strands of $D$
lie entirely below the height of the tops of the copies of $S$.
As in the proof of the previous lemma,
we can assume that
every strand connects two distinct copies of $S$,
or has at least one endpoint on the bottom edge of $D$.

Suppose there is a strand
that connects a non-adjacent pair of copies of $S$.
Between these copies of $S$
there must exist a copy of $S$
that is connected only to its two adjacent copies.
It must be connected to at least one of these
by at least $5$ strands.
The result now follows from the quadratic relation
and induction on the number of copies of $S$.

Now suppose every strand
either connects adjacent copies of $S$
or has at least one endpoint on the bottom edge of $D$.
If there is exactly one copy of $S$ in $D$ then we have $D = S$,
and we are done.
Suppose $D$ contains at least two copies of $S$.
Either the leftmost or rightmost copy of $S$
is connected to the bottom of $D$ by at most $5$ strands.
This copy of $S$
is connected to its only adjacent copy of $S$
by at least $5$ strands.
The result now follows from the quadratic relation
and induction on the number of copies of $S$.
\end{proof}

\begin{defn}
Suppose $X \in \mathcal{P}^0_n$ for some $n < 29$.
Then $X$ is a {\em morphism} from $P_n$ to the empty diagram if $XP_n = X$,
or equivalently,
if $XY = 0$ for every cup $Y \in \mathcal{P}^n_{n-2}$.
\end{defn}

\begin{lem}
\label{lem:morphism}
If $n < 29$ and $n \not \in \{0,10,18,28\}$
then every morphism from $P_n$ to the empty diagram is zero.
\end{lem}

\begin{proof}
First note that $\mathcal{P}_n$ is zero for odd values of $n$.
Thus we can assume $n = 2k$.
We have the following identities,
shown in the case $k=1$.
$$\labellist
  \tiny \hair 1pt
  \pinlabel $X$ at 50 55
  \endlabellist
  \raisebox{-18pt}{\includegraphics[scale=0.4]{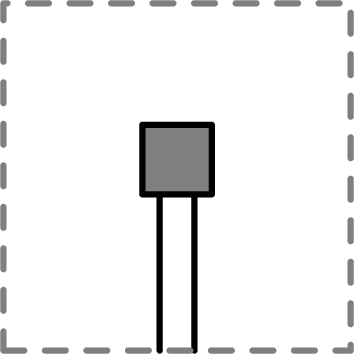}}
=
  \labellist
  \tiny \hair 1pt
  \pinlabel $X$ at 50 55
  \endlabellist
  \raisebox{-18pt}{\includegraphics[scale=0.4]{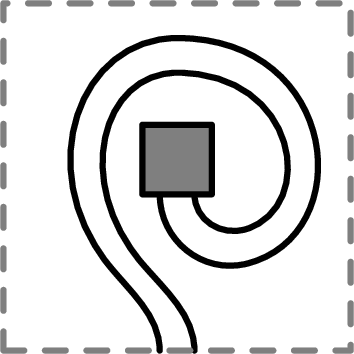}}
=
  \labellist
  \tiny \hair 1pt
  \pinlabel $X$ at 50 55
  \endlabellist
  \raisebox{-18pt}{\includegraphics[scale=0.4]{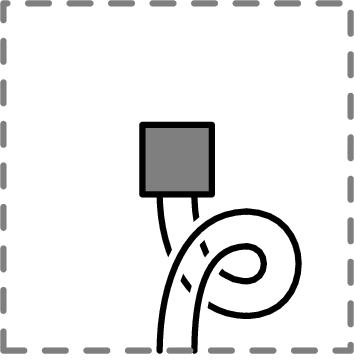}}
= q^{2k(k+1)}
  \labellist
  \tiny \hair 1pt
  \pinlabel $X$ at 50 55
  \endlabellist
  \raisebox{-18pt}{\includegraphics[scale=0.4]{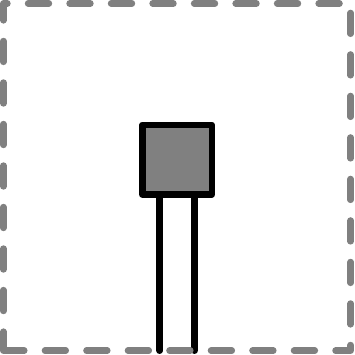}}
$$
The first equality is an isotopy.
The second follows from the braiding relation.
To prove the third,
resolve each crossing
and eliminate all terms that contain a cup attached to $X$.
There are $2k$ instances of a strand crossing itself.
Each of these contributes a factor of $iq^{3/2}$,
by Reidemeister one.
There are $2k(2k-1)$ instances of two distinct strands crossing.
Each of these contributes a factor of $iq^{1/2}$.

Thus we have $X = q^{2k(k+1)} X$.
If $X$ is non-zero then $q^{2k(k+1)} = 1$,
so $2k(k+1)$ is a multiple of $60$.
The result now follows from simple case checking.
\end{proof}

\begin{lem}
\label{lem:sixteen}
$\Phi$ is injective on $\mathcal{P}_n$ for $n \le 16$.
\end{lem}

\begin{proof}
The case $n = 0$ follows from Lemma~\ref{lem:zero},
and the case $n = 10$ follows from Lemma~\ref{lem:ten}.
Suppose $X \in \mathcal{P}^0_n$ is in the kernel of $\Phi$,
where $n \le 16$ and $n \not\in \{0,10\}$.
If $Y$ is a cup
then $XY$ is in the kernel of $\Phi$,
so $XY = 0$ by induction on $n$.
Thus $X$ is a morphism from $P_n$ to the empty diagram.
The result now follows from
Lemma~\ref{lem:morphism}.
\end{proof}

We are now ready to prove Theorem~\ref{thm:iso}.
Suppose $X \in \mathcal{P}_n$
is in the kernel of $\Phi$.
We must show $X = 0$.
We can assume $n > 16$.

Write $X$ as an element of $\mathcal{P}^{n-7}_7$.
The relation in Lemma~\ref{lem:seven} also holds in $\mathcal{P}$,
by Lemma~\ref{lem:sixteen}.
Thus $X$ is a linear combination of diagrams of the form $XAB$,
where $A \in \mathcal{P}^7_m$
and $B \in \mathcal{P}^m_7$
for some $m > 7$.
For any such $A$,
$XA$ is in the kernel of $\Phi$,
so $XA = 0$ by induction on $n$.
Thus $X = 0$.
This completes the proof that $\Phi$ is an isomorphism.

%----------------------------
\section{A basis}
\label{sec:basis}

We now define a set of diagrams that will form a basis for $\mathcal{P}_n$.

\begin{defn}
Let $\mathcal{B}_n$ be the set of diagrams of the form $XY$,
where $X$ is one of the four diagrams shown in Figure~ref{fig:Xs},
and $Y$ is a JW-reduced Temperley-Lieb diagram.
\end{defn}

The aim of this section is to prove the following.

\begin{thm}
\label{thm:basis}
$\mathcal{B}_n$ is a basis for $\mathcal{P}_n$.
\end{thm}

Before we prove Theorem~\ref{thm:basis},
we establish some additional consequences of the relations
on $\mathcal{P}$.

\begin{lem}
\label{lem:jointwo}
The diagram
$$ \labellist
     \tiny \hair 1pt
     \pinlabel $S$ at 60 50
     \pinlabel $\star$ [b] at 60 70
     \pinlabel $S$ at 140 50
     \pinlabel $\star$ [b] at 140 70
     \endlabellist
  \raisebox{-18pt}{\includegraphics[scale=0.4]{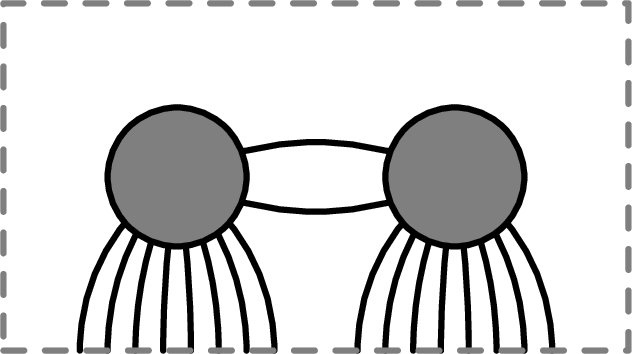}}$$
is a linear combination of diagrams
that have at most one copy of $S$.
\end{lem}

\begin{proof}
Let $\operatorname{Join}_2(S,S)$
denote the above diagram.
Consider $\operatorname{Join}_2(S,S) P_{16}$.
On the one hand,
this is zero by Lemma~\ref{lem:morphism}.
On the other hand,
we can write $P_{16}$ as a linear combination of Temperley-Lieb diagrams.
The identity diagram occurs with coefficient one.
Every other term contains a cup.
This cup either connects the two copies of $S$ or gives zero.
Thus $\operatorname{Join}_2(S,S)$
is equal to a linear combination of diagrams
that contain a copy of $\operatorname{Join}_3(S,S)$.

Now apply the same argument to $\operatorname{Join}_3(S,S) P_{14}$,
and then to $\operatorname{Join}_4(S,S) P_{12}$.
Thus we can work our way up to
a linear combination of diagrams
in which the two copies of $S$ are connected by five strands.
Now apply the quadratic relation
to obtain a linear combination of diagrams that have at most one copy of $S$.
\end{proof}

\begin{lem}
\label{lem:Xs}
If $m \in \{0,10,18,28\}$
then every morphism from $P_m$ to the empty diagram
is a scalar multiple of $XP_m$,
where $X$ is one of the four diagrams
shown in Figure~\ref{fig:Xs}.
\end{lem}

\begin{figure}[htb]
$$
% Empty diagram
     \includegraphics[scale=0.4]{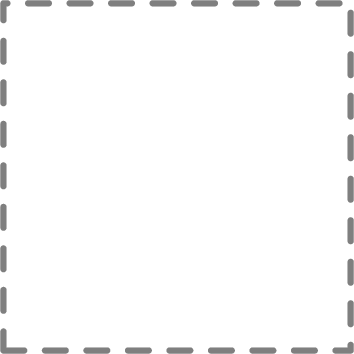},
\quad
% S
     \labellist
     \tiny \hair 1pt
     \pinlabel $S$ at 50 50
     \pinlabel $\star$ [b] at 50 70
     \endlabellist
     \includegraphics[scale=0.4]{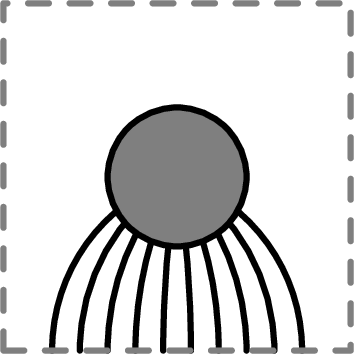},
\quad
% Join_1(S,S)
     \labellist
     \tiny \hair 1pt
     \pinlabel $S$ at 60 50
     \pinlabel $S$ at 140 50
     \pinlabel $\star$ [b] at 60 70
     \pinlabel $\star$ [b] at 140 70
     \endlabellist
     \includegraphics[scale=0.4]{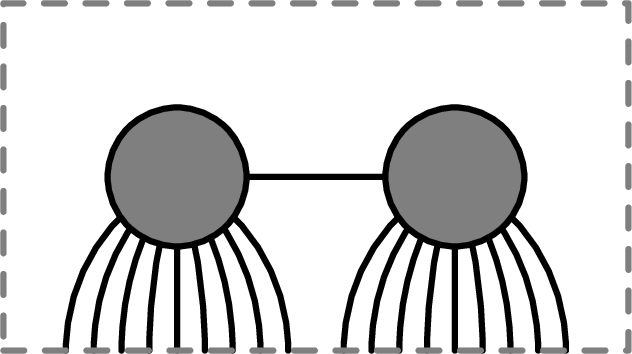},
\quad
% Join_{1,0}(S,S,S)
     \labellist
     \tiny \hair 1pt
     \pinlabel $S$ at 50 50
     \pinlabel $S$ at 130 50
     \pinlabel $S$ at 210 50
     \pinlabel $\star$ [b] at 50 70
     \pinlabel $\star$ [b] at 130 70
     \pinlabel $\star$ [b] at 210 70
     \endlabellist
     \includegraphics[scale=0.4]{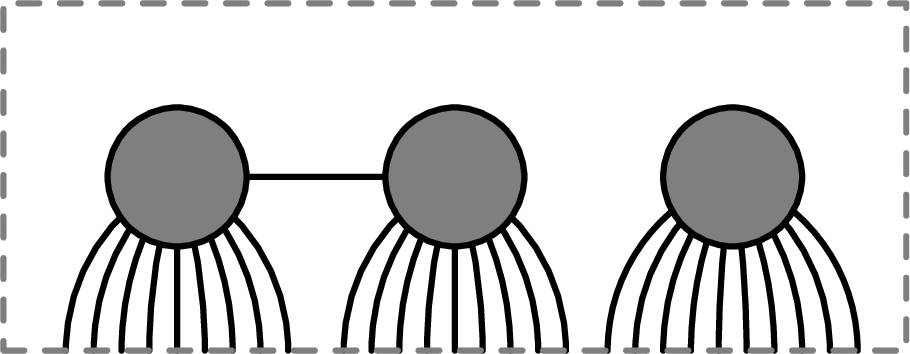}.
$$
\caption{Possible values for $X$}
\label{fig:Xs}
\end{figure}

\begin{proof}
Suppose $D$ is a diagram in $\mathcal{P}^0_m$.
We must show that $D P_m = X P_m$
where $X$ is one of the four diagrams shown in Figure~\ref{fig:Xs}.

We use induction on the number of copies of $S$ in $D$.
As in the proof of Lemma~\ref{lem:ten},
we can assume the copies of $S$
lie in a row at the top of $D$,
and all strands of $D$
lie entirely below the height of the tops of the copies of $S$.
If there is a strand with both endpoints on the same copy of $S$
then $D = 0$.
If there is a strand with both endpoints on the bottom edge of $D$
then $D P_m = 0$.
Thus we can assume every strand
either connects two copies of $S$,
or connects a copy of $S$ to the bottom edge of $D$.

By Lemma~\ref{lem:jointwo},
we can assume that any pair of copies of $S$
is connected by at most one strand.
If $m \in \{0,10,18\}$
then the only possible values of $D$
are as shown in Figure~\ref{fig:Xs}.
If $m = 28$ then there is one other possibility,
namely
that $D$ has three copies of $S$,
but the second and third copy are connected by a strand,
instead of the first and second.
Using the braiding relation,
we can bring the leftmost copy of $S$ to the right.
Thus $D = D' \beta$
where $D'$ is the fourth diagram in Figure~\ref{fig:Xs},
and $\beta$ is a braid.
Then $D' \beta P_{28}$
is a scalar multiple of $D' P_{28}$.
\end{proof}

\begin{lem}
\label{lem:X}
Any diagram $X \in \mathcal{P}^0_n$
is a linear combination of diagrams
that either contain a cap
or are one of the four diagrams in Figure~\ref{fig:Xs}.
\end{lem}

\begin{proof}
Suppose $n \ge 29$.
Let $P_{29} \otimes \id_{n-29}$
denote the Jones-Wenzl idempotent
with extra vertical strands if necessary
to bring the total up to $n$.
Consider $X(P_{29} \otimes \id_{n-29})$.
This is zero by Lemma~\ref{lem:jw}.
We can use this to write $X$
as a linear combination of diagrams that contain a cap.

Now suppose $n \le 29$
but $n \not\in \{0,10,18,28\}$.
Consider $XP_n$.
On the one hand,
this is zero by Lemma~\ref{lem:morphism}.
On the other hand,
we can write $P_n$ as a linear combination of Temperley-Lieb diagrams.
We can use this to write $X$
as a linear combination of diagrams that contain a cap.

Now suppose $n \in \{0,10,18,28\}$.
Consider $X P_n$.
On the one hand,
by Lemma~\ref{lem:Xs},
$X P_n$ is a scalar multiple of $D P_n$,
where $D$ is one of the four diagrams in Figure~\ref{fig:Xs}.
On the other hand,
we can write $P_n$ as a linear combination of Temperley-Lieb diagrams.
Thus $X$ is a scalar multiple of $D$,
modulo terms that contain a cap.
\end{proof}

\begin{lem}
\label{lem:Y}
Any Temperley-Lieb diagram $Y$ in $\mathcal{P}^m_n$
is a linear combination of diagrams
that either contain a cup
or are JW-reduced.
\end{lem}

\begin{proof}
Let $x_1,\dots,x_n$
be the endpoints at the bottom of $Y$,
in order from left to right.
Let $a_i$ be the number of strands
that have one endpoint at or to the left of $x_i$,
and the other endpoint to the right of $x_i$ or at the top of $Y$.
Let $a_0 = 0$.
Call $(a_0,\dots,a_n)$ the {\em sequence corresponding to $Y$}.
This sequence satisfies $a_{i+1} = a_i \pm 1$.

Suppose $a_i \ge 29$ for some $i$.
Then $a_i = 29$ for some $i$.

Let $L$ be a vertical line
between $x_i$ and $x_{i+1}$.
Assume
all the endpoints at the top of $Y$
are to the right of $L$,
and $L$ intersects the strands of $Y$ in exactly $29$ places.

Let $Y'$ be the result of inserting
a sideways copy of $P_{29}$ into $L$.
This is zero by Lemma~\ref{lem:jw}.
We can use this to write $Y$
as a linear combination of diagrams
obtained by inserting non-identity Temperley-Lieb diagrams into $L$.
Each diagram in this linear combination
either contains a cup,
or has a corresponding sequence
that is smaller than $(a_0,\dots,a_n)$
in lexicographic order.

This process must terminate
with a linear combination of Temperley-Lieb diagrams
of the desired form.
\end{proof}

\begin{lem}
$\mathcal{B}_n$ spans $\mathcal{P}_n$.
\end{lem}

\begin{proof}
Suppose $D$ is a diagram in $\mathcal{P}_n$
Draw $D$ in a rectangle,
with all endpoints on the bottom edge.

As in the proof of Lemma~\ref{lem:ten},
we can assume the copies of $S$
lie in a row at the top of $D$,
and every strand either connects adjacent copies of $S$
or connects a copy of $S$ to the bottom edge of $D$.
Let $m$ be the number of strands
that connect a copy of $S$ to the bottom edge of $D$.
We proceed by induction on $m$.

Write $D$ in the form $XY$,
where $X$ is a diagram in $\mathcal{P}^0_m$
and $Y$ is a Temperley-Lieb diagram in $\mathcal{P}^m_n$.
Now apply Lemma~\ref{lem:X} to $X$
and Lemma~\ref{lem:Y} to $Y$.
Thus $XY$ is a linear combination of diagrams of the form $X'Y'$,
where $X$ either contains a cap
or is one of the four diagrams in Figure~\ref{fig:Xs},
and $Y$ either contains a cup
or is JW-reduced.

If $X'$ contains a cap
or $Y'$ contains a cup
then $X'Y'$ has a smaller value of $m$.
If $X'$ does not contain a cap
and $Y'$ does not contain a cup
then $X'Y'$ is an element of $\mathcal{B}_n$.
\end{proof}

By Lemma~\ref{lem:dimension},
the dimension of $\mathcal{P}_n$
is the number of elements of $\mathcal{B}_n$.
This completes the proof the $\mathcal{B}_n$
is a basis for $\mathcal{P}_n$.

%----------------------------------------------
\section{The other $ADE$ planar algebras}
\label{sec:ade}

We now consider the subfactor planar algebras
of types $A_N$, $D_{2N}$ and $E_6$.
Each of these has a presentation and a basis
similar to those we gave for $E_8$,
and by similar arguments.

The case of $A_N$ is already well understood.
In the case of $D_{2N}$,
a presentation and a basis were given in \cite{mps}.
However our methods give a different point of view.
The basis elements in \cite{mps}
are complicated linear combinations of diagrams
built out of minimal idempotents,
whereas each of our basis elements is a single diagram.
This is not to say our basis is better.
Indeed,
the apparently more complicated basis
is actually more natural from a purely algebraic point of view.

Our presentations for the planar algebras are as follows.

The subfactor planar algebra with principal graph $A_N$
is the planar algebra with no generators
and the defining relations
\begin{itemize}
\item $\raisebox{-6pt}{\includegraphics[scale=0.4]{manfig1}} = [2]  \,
       \raisebox{-6pt}{\includegraphics[scale=0.4]{manfig2}}$.
\item $P_N = 0$,
\end{itemize}
where $q = e^{i \pi / (N+1)}$.

The subfactor planar algebra with principal graph $D_{2N}$
is the planar algebra $\mathcal{P}$
with a single generator $S \in \mathcal{P}_{4N-4}$
and defining relations
\begin{itemize}
\item $\raisebox{-6pt}{\includegraphics[scale=0.4]{manfig1}} = [2]  \,
       \raisebox{-6pt}{\includegraphics[scale=0.4]{manfig2}}$.
\item $\rho(S) = \sqrt{-1} S$.
\item $\tau(S) = 0$.
\item $S^2 = P_{2N-2}$.
\item $P_{4N-3} = 0$,
\end{itemize}
where $q = e^{i \pi / (4N-2)}$.
Note that we cannot use a defining relation $\hat{S} P_{4N-2} = 0$,
since $P_{4N-2}$ is not defined for the above value of $q$.

The subfactor planar algebra with principal graph $E_6$
is the planar algebra $\mathcal{P}$
with a single generator $S \in \mathcal{P}_6$
and the defining relations
\begin{itemize}
\item $\raisebox{-6pt}{\includegraphics[scale=0.4]{manfig1}} = [2]  \,
       \raisebox{-6pt}{\includegraphics[scale=0.4]{manfig2}}$,
\item $\rho(S) = e^{4 i \pi / 3} S$,
\item $\tau(S) = 0$,
\item $S^2 = S + [2]^2[3] P_3$,
\item $\hat{S} \, P_8 = 0$,
\end{itemize}
where $q = e^{i \pi/12}$.

Each of these planar algebras satisfies some kind of ``braiding relation''.
In the $A_N$ case,
Reidemeister moves two and three say
you can drag a strand over or under any part of a diagram.
In the $D_{2N}$ case,
\cite{mps} prove you can drag a strand over any part of a diagram,
and you can drag a strand under any part of the diagram
up to a possible change of sign.
In the $E_6$ and $E_8$ cases,
you can drag a strand over any part of a diagram,
but you cannot drag a strand under a generator,
even up to sign.

In the definition of JW-reduced,
we must replace the number $28$
with the number $k$ such that $P_{k+1} = 0$.
In the $E_6$ case,
$P_{11} = 0$.

Each of these planar algebras
has a basis similar to the one defined in Section~\ref{sec:basis}.
The basis elements are of the form $XY$,
where $X$ is one of a short list of possibilities,
and $Y$ is a JW-reduced Temperley-Lieb diagram.
In the case of $A_N$,
$X$ is simply the empty diagram.
In the cases $D_{2N}$ and $E_6$,
$X$ is either the empty diagram
or $S$,
with all strands pointing down.

%--------------------------------------------------

\newcommand{\etalchar}[1]{$^{#1}$}
\providecommand{\bysame}{\leavevmode\hbox to3em{\hrulefill}\thinspace}
\providecommand{\MR}{\relax\ifhmode\unskip\space\fi MR }
\providecommand{\MRhref}[2]{%
  \href{http://www.ams.org/mathscinet-getitem?mr=#1}{#2}
}
\providecommand{\href}[2]{#2}

\end{document}